\pdfoutput=1
\documentclass[11pt]{article}
\usepackage{vmargin}
\usepackage{times}

\usepackage[utf8]{inputenc}
\usepackage[T1]{fontenc}

\usepackage[whole]{bxcjkjatype}
\usepackage{CJK}

\usepackage{comment}

\usepackage{epsfig}
\usepackage{graphicx}
\usepackage{amssymb}
\usepackage{amsmath}
\usepackage[utf8]{inputenc}
\usepackage{amsfonts}
\usepackage{amsthm}
\usepackage{paralist}
\usepackage{stmaryrd}
\usepackage[toc,page]{appendix}
\usepackage{bbm}
\usepackage{bm}

\usepackage{mathrsfs}

\usepackage{longtable}
\usepackage[labelfont=bf,font=it]{caption}

\usepackage[colorlinks=true,linkcolor=blue,citecolor=blue,urlcolor=blue,pdfborder={0 0 0}]{hyperref}
\usepackage{bookmark}

\usepackage{color}

\usepackage[dvipsnames]{xcolor}

\usepackage{centernot}

\usepackage{cite}
\usepackage{cleveref}
 \usepackage{crossreftools}

  \usepackage{microtype}

\usepackage{mathtools} 

\theoremstyle{plain}
\newtheorem{corollary}{Corollary}
\newtheorem{theorem}[corollary]{Theorem}
\newtheorem{lemma}[corollary]{Lemma}

\newtheorem*{conjecture*}{Conjecture}
\newtheorem{proposition}[corollary]{Proposition}
\newtheorem*{theorem*}{Theorem}
\newtheorem*{lemma*}{Lemma}
\newtheorem*{definition*}{Definition}
\newtheorem*{corollary*}{Corollary}
\theoremstyle{definition}

\theoremstyle{remark}

\newcommand{\R}{\mathbb{R}}

\newcommand{\Z}{\mathbb{Z}}
\newcommand{\proba}{\mathbb{P}}

\newcommand{\E}{\mathbb{E}}

\renewcommand{\S}{\mathcal{S}}

\newcommand{\Comp}{\mathfrak{C}}

\newcommand{\1}{\mathbf{1}}

\newcommand{\off}{\hbox{{\rm off}}}

\newcommand{\be}{\begin{eqnarray}}
\newcommand{\ee}{\end{eqnarray}}

\DeclareMathOperator{\Cov}{Cov}

\newcommand{\Bc}{\mathcal{B}}
\newcommand{\Gr}{\mathbb{G}}
\newcommand{\Fc}{\mathcal{F}}
\newcommand{\Gc}{\mathcal{G}}
\newcommand{\Hc}{\mathcal{H}}
\newcommand{\Sc}{\mathcal{S}}
\newcommand{\Tf}{\mathfrak{T}}

\newcommand{\GrSeq}{\Gr_{\text{seq}}}

\DeclareMathOperator*{\slr}{\leftrightarrow}
\DeclareMathOperator*{\snlr}{\nleftrightarrow}

\begin{document}
\title{Notes on Kozma and Nitzan's proof of $\theta(p_c)=0$.}
\author{Arthur Blanc-Renaudie}
\maketitle
\begin{abstract}
We present a human intelligible proof of the conjectures of Kozma and Nitzan, that they used to show $\theta(p_c)=0$. The proof was recently found by AI, but was very poorly written \cite{Leder26,KN-AI26}. 
Our version is much shorter as it completely skips the horizontal term and its long estimate.
\end{abstract}
\section{Introduction}
\subsection{Motivations}
The most famous problem in probability theory has been for a very long time $\theta(p_c)=0$: take $\Z^d$, $d\geq 2$, keep each edge with probability $p$, and erase the others to get a random graph.
There exists a critical probability $p_c$ such that a.s. for $p<p_c$, all the clusters (connected components) are finite, and for $p>p_c$ there exists an infinite cluster. The big conjecture is that all critical clusters are finite.
This was solved for $d=2$ \cite{Harris60,Kesten80}, using planarity, and $d\geq 11$ \cite{FH17} using high-dimensional methods.

Recently, Kozma and Nitzan \cite{KN24} gave a conditional proof of the conjecture in dimensions $d\geq 3$.
Their proof is based on a one-step renormalisation scheme from \cite{BGN91,GM90}. 
To summarize the proof, through a well-chosen exploration and construction, they show that if percolation occurs in $\Z^d$ at $p_c$, then it also occurs in a slab $\Z^2\times[-n,n]^{d-2}$, contradicting \cite{AG91}.
However, their proof relied on a conjectured correlation inequality that was doubted and thought very hard to show.

It was a huge shock when last month AI proved this inequality. 
However, its announcement was a complete disaster.
It all started with a leak of 100,000 lines of Lean \cite{LederLean26}, with a manuscript \cite{Leder26} that clearly does not respect the standards of a mathematical paper.
Bou Rabee then prompted an AI to explain and extend \cite{KN-AI26} the proof, again hard to read. The files were then extended in \cite{BRL26,PercolationNotes26}.
All those poorly written files left readers with considerable work just to check the proof.
Even the specialists who checked the proof were unsure about the details or why the proof works as a whole.

The goal of this paper is to provide a human-written proof accessible to most probabilists. It is reconstructed from the AI's ideas and differs from the original proof.
Notably, it has its own notation and explanations and avoids the "horizontal term", shortening the proof by almost half.

\textbf{AI statement.} AI has been extensively used to understand AI awful work, and for proofreading. 
\subsection{Main results}
We consider Bernoulli bond percolation on an arbitrary finite graph $\Gr$ such that each edge $e$ is independently open with probability $p_e\in [0,1]$. To avoid division by $0$, we assume that every edge satisfies $p_e\in (0,1)$, as the main results can be recovered by continuity with respect to $(p_e)_{e\in \Gr}$. 

We say that $x,y$ are connected and write $x\slr y$ if there is a path of open edges between them, and otherwise write $x\snlr y$. 
Let $\Comp$ be the set of clusters (connected subgraphs) of $\Gr$. For $x\in \Gr$, let $C_x$ be its cluster in the random percolated graph. For $S\subseteq \Gr$, let $C_S$ denote the set of vertices connected to $S$.  
For a random variable $X$ and an event $E$ we write $\E[X;E]=\E[X\1_E]$. 

We show the next result, which is a stronger version of Conjecture 4 of Kozma and Nitzan \cite{KN24}.
\begin{theorem} \label{thm:Main} Let $0,b\in \Gr$. Let $A\subseteq \Gr$. Let $\Fc:\Comp\mapsto \R$ be nondecreasing. Let $a$ be the vertex minimizing $\E[\Fc(C_a)]$ in $A$, chosen according to an arbitrary fixed order on $\Gr$ if equality. We have
\[ \E[\Fc(C_0); 0\slr A]\geq \E [\Fc(C_a) ; 0\slr A].\]
\end{theorem}
In particular, by taking $F:C\mapsto \1_{b\in C}$ and using Harris inequality, one recovers Conjectures 1 and 2 of Kozma and Nitzan, as well as Question 7.
\[ \proba(0\slr b,0\slr A)\geq \min(\proba(a\slr b,0\slr A):a\in A) \geq \proba(0\slr A)\min(\proba(a\slr b):a\in A) . \]
Kozma and Nitzan only used their weakest conjecture, Conjecture 3, which is a weaker version of Conjecture 1, in their proof of $\theta(p_c)=0$. We do not adress Conjecture 6 and Questions 5 and 9, which was deduced from Theorem \ref{thm:Main} by Bou-Rabee \cite{KN-AI26}. Question 8 has a counterexample in \cite{KN-AI26}.

We write the difference of the two sides of Theorem \ref{thm:Main} as $\Delta_{A,\Fc}(0)$. 
We prove Theorem \ref{thm:Main} by induction on $|A|$, by removing vertices one by one. Removing the vertex $b$ from $A\cup\{b\}$ corresponds to an operator $\pi_{A,b}$ defined as follows. 
 For every $\Fc:\Gr\mapsto \R$, for every $A\subseteq \Gr$, $x\in \Gr, b\notin A$, let 
\[ \pi_{A,b}\Fc(x):=\Fc(x)-\Fc(b)\proba(x\slr b |b\snlr A).\]
To remove several vertices, for a sequence $\Bc=(b_1,\dots,b_k)$ of different vertices not in $A$, we use
\[ \pi_{A,\Bc}:=\pi_{A\cup\{b_1,\dots, b_{k-1}\},b_k}\circ \dots \circ \pi_{A\cup \{b_1\},b_2} \circ \pi_{A,b_1}. \] 
When $\Bc$ is the empty sequence, written $\emptyset$, we define by convention $\pi_{A,\emptyset}$ to be the identity operator.  The order of the operators may seem counter-intuitive, but one may see $b_1,b_2,\dots, b_k$ as the vertices we have already removed from $A$. 
Because of that point of view, our induction requires us to show the following stronger result.
Let $\GrSeq$ be the set of sequences $(b_1, b_2,\dots,b_k)$ of different vertices of $\Gr$, possibly empty. With a slight abuse of notation, we may view $\GrSeq$ as the set of subsets of $\Gr$ and use the usual set operations with it. Also, if $\Bc_1,\Bc_2\in \GrSeq$ are disjoint, we write the natural concatenation of $\Bc_1$ followed by $\Bc_2$ as $\Bc_1\uplus \Bc_2\in \GrSeq$.

\begin{theorem} \label{thm:Main_with_removal} Let $A\subseteq \Gr$. Let $\Bc\in \GrSeq$ with $\Bc\cap A=\emptyset$. We have $\pi_{A,\Bc}\Delta_{A,\Fc}\geq 0$ pointwise on $\Gr \backslash (A\cup \Bc)$ for every nondecreasing $\Fc: \Comp \mapsto \R$.
\end{theorem} 
Note that when $\Bc=\emptyset$, Theorem \ref{thm:Main_with_removal} states that $\Delta_{A,\Fc}\geq 0$, which is exactly Theorem \ref{thm:Main}. This deals with the case $0\notin A$. The complementary case, where $0\in A$, is essentially trivial. 

The technical part of the proof is to show the next result. We defer its proof to Section \ref{Sec:TechnicalPart}, so that we can first see how to use it. For every $\Fc:\Comp\mapsto \R$, $A\subseteq \Gr$, $b\in \Gr\backslash A$, $x\in \Gr$, let
\[ [\kappa_{b,A}\Fc](x):=\Cov \left (\Fc(C_b);\1_{x\slr b} \middle |b\snlr A\right ).\]
\begin{theorem} \label{thm:Projection_Covariance}For every nondecreasing $\Fc:\Comp\mapsto \R$, every $\Bc\in \GrSeq$, $A\subseteq \Gr$ with $A\cap \Bc=\emptyset$ and $b\in \Gr \backslash (A\cup \Bc)$, we have $[\pi_{A\cup \{b\},\Bc}\circ \kappa_{b,A}]\Fc\geq 0$ pointwise on $\Gr\backslash (A\cup \Bc\cup\{b\})$. 
\end{theorem} 
Theorem \ref{thm:Projection_Covariance} morally says that after conditioning on $b\snlr A$ and removing the vertices of $\Bc$, the random variables $\Fc(C_b)$ and $\1_{x\slr b}$ remain positively correlated.
We also use the next corollary of Theorem \ref{thm:Projection_Covariance}, which morally says that $\Fc(C_a)$ and $\1_{x\slr b}$ stay negatively correlated after removing $\Bc$. For every $\Fc:\Comp\mapsto \R$, $A\subseteq \Gr$, $a\in A$, $b\in \Gr\backslash A$, $x\in \Gr$, let
\[ [\kappa_{a,b,A}\Fc](x):=\Cov \left (\Fc(C_a);\1_{b\slr x} \middle |b\snlr A\right ).\]
 For a set or cluster $S\subseteq \Gr$, let $\proba_{\off S}$, $\E_{\off S}$, $\Cov_{\off S}$ be, respectively, the probability measure, expectation, and covariance after forcing all edges adjacent to $S$ to be closed.
\begin{corollary} \label{cor:Projection_Covariance} For every nondecreasing $\Fc:\Comp\mapsto \R$, every $\Bc\in \GrSeq$, $A\subseteq \Gr$ with $A\cap \Bc=\emptyset$, $a\in A$, and $b\in \Gr \backslash (A\cup \Bc)$, we have $[\pi_{A\cup \{b\},\Bc}\circ \kappa_{a,b,A}]\Fc\leq 0$ pointwise on $\Gr\backslash (A\cup \Bc\cup\{b\})$. 
\end{corollary}
\begin{proof} Using Aizenman's off method, we may rewrite $\kappa_{a,b,A}$ as 
\[ [\kappa_{a,b,A}\Fc](x) =\Cov \left (\E[\Fc(C_a)|C_b];\1_{b\slr x} \middle |b\snlr A\right )=\Cov \left (\E_{\off C_b}[\Fc(C_a)];\1_{b\slr x} \middle |b\snlr A\right ).
\]
The result follows by applying Theorem \ref{thm:Projection_Covariance} with $\Fc':C\mapsto -\E_{\off C}[\Fc(C_a)]$, which is nondecreasing. Indeed for larger $C$, we remove more edges, so after removal $C_a$ and $\Fc(C_a)$ decreases.
\end{proof}

\subsection{Some remarks}
One may notice that the operators $\pi_{A,b}$ are actually orthogonal projectors. Indeed, for every $A\subseteq \Gr$, one may consider the kernel defined by $K_A(u,v):= \proba(u\slr v,u\snlr A)$. This kernel appeared in earlier percolation papers and is positive definite on $\Gr\backslash A$. Indeed, for every $\Fc:\Gr\mapsto \R$, we have 
\[ \sum_{u,v} \Fc(u)K_A(u,v)\Fc(v)=\E\Big [\sum_{C,C\cap A=\emptyset} \Big ( \sum_{u\in C} \Fc(u) \Big )^2 \Big ]\geq 0,\]
 where we sum over clusters $C$ of $\Gr$. Using some linear algebra, one can show that $\pi_{A,b}$ is the orthogonal projector onto the set of functions with $\Fc(b)=0$ in the Hilbert space with reproducing kernel $K_A$. These Hilbert spaces were already well known.
 
 The fact that the operators are orthogonal projections is probably the main reason why the calculation works so well when working with covariance. However, we could not make use of this intuition to write the proof, as the main difficulty is that the kernel changes with each projection. We have the intuition that those changes of space are the main mechanism that produces the extra covariance factor throughout the proof, but we also could not use it.
 
Another reason why the computations works so well is Aizenman's off method, which allows to transfer many conditioning to another percolation model where some edges are forced to be closed. Morally after revealing some clusters, the other clusters behave as if all the edges adjacent to the revealed part would be originally closed.
 This is one of the most powerful tools to study Benoulli percolation, expecially in high dimensions. (See e.g. \cite{BR26} for a recent formalism and extended use.)
 
The proof works verbatim if one also consider hyperedges in $\Gr$ that can connect several vertices at once. This allows to deduce the same correlations inequalities for site percolation, by considering the dual graphs where sites now becomes hyperedges, and edges becomes vertices. 
Then Kozma and Nitzan's proof of $\theta(p_c)=0$ still works for site percolation on $\Z^d$, $d\geq 3$. (See Bou Rabee \cite{PercolationNotes26,BRL26})
\section{Inductive proof of Theorem \ref{thm:Main_with_removal} given Theorem \ref{thm:Projection_Covariance}} \label{Sec:First_Induction}
Fix $\Fc$. We prove that the result holds by induction on $|A|$ for every $\Bc\in \GrSeq$ with $\Bc\cap A=\emptyset$. First, if $|A|\leq 1$, then we directly have $\Delta_{A,\Fc}=0$. So let us assume $|A|\geq 2$. Take $a$ as in Theorem \ref{thm:Main}. Let $b\neq a \in A$ be a vertex in $A$ maximizing $\E[\Fc(C_b)]$. We aim to remove $b$ from the set $A$. Assume by the induction hypothesis that Theorem \ref{thm:Main_with_removal} holds when $A$ is replaced by $A\backslash \{b\}$. We compare $\Delta_{A,\Fc}$ with $\Delta_{A\backslash \{b\},\Fc}$. Fix $x\in \Gr\backslash (A\cup \Bc)$.
 To this end, by splitting the expectation defining $\Delta_{A,\Fc}(x)$ according to whether $x\slr A\backslash b$ holds or not, we get 
\begin{align*} \Delta_{A,\Fc}(x) & =\Delta_{A\backslash \{b\},\Fc}(x)+ \E[\Fc(C_x)-\Fc(C_a);x\slr b;b\snlr A\backslash \{b\}]
  \\ & =\Delta_{A\backslash \{b\},\Fc}(x)+ \E[\Fc(C_b)-\Fc(C_a);x\slr b|b\snlr A\backslash \{b\}]\proba(b\snlr A\backslash \{b\}).\end{align*}
using $C_x=C_b$ when $x\slr b$, and $\proba(b\snlr A\backslash \{b\})\neq 0$ for the last equality.

We may rewrite the last expectation as 
\begin{equation} \Cov(\Fc(C_b)-\Fc(C_a),\1_{x\slr b}|b\snlr A\backslash \{b\})+\E[\Fc(C_b)-\Fc(C_a)|b\snlr A\backslash \{b\}]\proba(x\slr b|b\snlr A\backslash\{b\}).
  \label{eq:Leftover_Main_Induction}
\end{equation}
We need to understand how $\pi_{A,\Bc}$ operates on the above display when seen as a function of $x$. Note that $\pi_{A,\Bc}$ is a linear operator, so we may deal with each of the terms separately. First, using the definition of $\kappa$, we recognize the covariance term as $[(\kappa_{b,A\backslash \{b\}}-(\kappa_{a,b,A\backslash \{b\}}))\Fc](x)$, and Theorem \ref{thm:Projection_Covariance} with Corollary \ref{cor:Projection_Covariance} says that applying $\pi_{A,\Bc}$ gives something non-negative.
Next, since $x\notin A$, we have $|C_b|\geq 2$ whenever $x\slr b$. Thus, the probability in the last display, when seen as a function of $x$, is proportional to  
\[ \Cov(\1_{|C_b|\geq 2}; \1_{x\slr b}|b\snlr A\backslash \{b\})=[\kappa_{b,A\backslash \{b\}}\Gc](x),\]
where $\Gc:C\mapsto \1_{|C|\geq 2}$. So by Theorem \ref{thm:Projection_Covariance}, applying $\pi_{A,\Bc}$ to the probability in \eqref{eq:Leftover_Main_Induction} gives something nonnegative. In other words, by linearity of $\pi_{A,\Bc}$, to get a lower bound on $ \Delta_{A,\Fc}(x)$, we may replace the expectation in \eqref{eq:Leftover_Main_Induction} by a lower bound. Moreover, using $\E[\Fc(C_a)]\leq \E[\Fc(C_b)]$, we get
\[ \E[\Fc(C_b)-\Fc(C_a);b\snlr A\backslash \{b\}]\geq -\E[\Fc(C_b)-\Fc(C_a);b\slr A\backslash \{b\}] = -\Delta_{A\backslash \{b\},\Fc}(b).\]

In conclusion, applying $\pi_{A,\Bc}$ and using linearity gives
\[  [\pi_{A,\Bc} \Delta_{A,\Fc}](x) \geq  \pi_{A,\Bc}\Big [  \Delta_{A\backslash \{b\},\Fc}-\Delta_{A\backslash \{b\},\Fc}(b) \cdot \Big (y\mapsto \proba(y\slr b|b\snlr A\backslash\{b\} )\Big )\Big ](x). \]
We recognize the above function inside the bracket as exactly the operator $\pi_{A\backslash\{b\},b}$ applied to the function $\Delta_{A\backslash \{b\},\Fc}$. Therefore, we may rewrite the last inequality as 
\[  [\pi_{A,\Bc} \Delta_{A,\Fc}] (x) \geq  [\pi_{A,\Bc}\circ \pi_{A\backslash\{b\},b}]\Delta_{A\backslash \{b\},\Fc}(x)=[\pi_{A\backslash\{b\},(b)\uplus\Bc}]\Delta_{A\backslash \{b\},\Fc}(x)\geq 0, \]
using the induction hypothesis for the last inequality, which concludes the proof.
\section{Proof of Theorem \ref{thm:Projection_Covariance}} \label{Sec:TechnicalPart}
We prove by induction on $|\Bc|$ that Theorem \ref{thm:Projection_Covariance} holds for every $A\subseteq \Gr$ with $A\cap \Bc=\emptyset$, and every $\Fc$ nondecreasing. Fix $\Bc=(b_1,\dots, b_k)$, possibly empty, and assume that the theorem holds for every $\Bc'$ with $|\Bc'|<|\Bc|$. Fix $A\subseteq \Gr$ with $A\cap \Bc=\emptyset$. The proof is divided into two parts.  First, we show a conditioned sub-theorem (Theorem \ref{thm:PostOff}), using the induction hypothesis. Then, using a short Gibbs argument, we deduce Theorem \ref{thm:Projection_Covariance} by applying Theorem \ref{thm:PostOff} to many functions $\Fc$.
\subsection{The inductive part}
 We first prove the next result, which one may see as a post-off-method statement. 
\begin{theorem} \label{thm:PostOff} Let 
\[ \Sc:x\mapsto \E[\Cov_{\off C_A}(\Fc(C_b);\1_{x\slr b});b\snlr A].
\] 
For every $x\notin A\cup \Bc\cup \{b\}$, with $\Fc$ nondecreasing, we have $[\pi_{A\cup \{b\},\Bc}\Sc](x)\geq 0$. 
\end{theorem} 
When $\Bc=\emptyset$, Harris's inequality directly gives $\Sc(x)\geq 0$, so we may assume $|\Bc|\geq 1$. 

The main idea behind the proof of Theorem \ref{thm:PostOff} is that by linearity, instead of viewing $\pi_{A\cup \{b\},\Bc}$ as acting on $\Sc$, we may look at its action on $\1_{x\slr b}$. Since the operator $\pi_{A\cup \{b\},\Bc}$ appeared when we removed the vertices $b_1,b_2,\dots, b_k$, we may hope that applying it will somehow add them back. This is exactly what happens as its action produces, in order, $\1_{x\slr \{b,b_1\}}$,\dots, $\1_{x\slr \{b,b_1,\dots, b_k\}}$, with some correction terms. Naturally, the correction terms are similar to those found in the proof of Theorem \ref{thm:Main_with_removal}, since the proof goes in the opposite direction. This allows us to rewrite the corrections as covariance terms with the correct form and deal with them using the induction hypothesis. 

Beforehand, let us rewrite $\Sc$ in a more convenient form. 
\begin{lemma} \label{lem:Rewrite_Sc0} Let 
\[ \Theta:=\1_{b\snlr A} (\Fc(C_b)-\E_{\off C_A}[\Fc(C_b)]). \]
 For every $x\in \Gr$, we have 
\[ \S(x)=\E[\Theta\1_{x\slr b}] .\]
\end{lemma} 
\begin{proof} By splitting according to the value of $C_A$ and using a standard covariance formula, we have 
\[ \Sc(x)=\sum_{A\subseteq K\subseteq \Gr, b\notin K} \proba(C_A=K) \E_{\off K} [(\Fc(C_b)-\E_{\off K}[\Fc(C_b)])\1_{x\slr b}; b\snlr A].\]
The law of total probability then gives 
\[ \Sc(x)=\E[\1_{b\snlr A}(\Fc(C_b)-\E_{\off C_A}[\Fc(C_b)])\1_{x\slr b}]=\E[\Theta \1_{x\slr b}].\qedhere\]
\end{proof} 
One may understand $\Theta$ as the fluctuations of $\Fc(C_b)$ that remain after conditioning on $C_A$.
\begin{lemma} \label{lem:Off_Fluctuations} We have a.s. $\E[\Theta|(C_r)_{r\in A}]=0$. 
\end{lemma} 
\begin{proof} This is an immediate consequence of Aizenman's off method. 
\end{proof} 

Then, for $1\leq i \leq k$, write 
\[ \Sc_i:x\mapsto \E[\Theta \1_{x\slr \{b,b_1,\dots, b_i\}}] \]
 and $\Sc_0=\Sc$. And let 
 \[ \Phi:K\subseteq \Gr \mapsto -\E_{\off K}[\Theta]. \]
  The following result states that iteratively applying the different operators of $\pi_{A\cup \{b\},\Bc}$ to $\Sc$ produces, in order, $\Sc_1,\Sc_2,\dots$ with some corrections that are covariances with, a priori, the right form.
\begin{lemma}\label{lem:OneStep_Induction} For every $0\leq i <k$, we have
 \[ [\pi_{A\cup \{b,b_1,\dots,b_i\},b_{i+1}}\Sc_i](x)=\Sc_{i+1}(x)+\proba(b_{i+1}\snlr A\cup \{b,b_1,\dots,b_i\})[\kappa_{b_{i+1},A\cup\{b,b_1,\dots, b_i\}}\Phi](x).\]
\end{lemma}  
\begin{proof} By linearity, we may look at the action of $\pi_{A\cup \{b,b_1,\dots,b_i\},b_{i+1}}$ on $x\mapsto \1_{x\slr \{b,b_1,\dots, b_i\}}$. To ease the notation, let $B_i=\{b,b_1,\dots, b_i\}$ with convention $B_0=\{b\}$. By definition of $\pi$, we have
\[ [\pi_{A\cup B_i,b_{i+1}} \1_{\cdot \slr B_i}](x)=\1_{x\slr B_i}-\1_{b_{i+1}\slr B_i}\proba(x\slr b_{i+1}|b_{i+1}\snlr A\cup B_i) .\]
We then make appear the expected term $\1_{x\slr B_{i+1}}$ and split the rest of the formula into a constant factor and a centered factor. We do so to make appear another covariance later. Using the identity $\1_{x\slr B_{i+1}}= \1_{x\slr B_i}+ \1_{b_{i+1}\snlr B_i} \1_{x\slr b_{i+1}}$, we get
\begin{align}  [\pi_{A\cup B_i,b_{i+1}} \1_{\cdot \slr B_i}](x)=\1_{x\slr B_{i+1}}- & \proba(x\slr b_{i+1}|b_{i+1}\snlr A\cup B_i)  \notag \\ 
- & \1_{b_{i+1}\snlr B_i} (\1_{x\slr b_{i+1}} -\proba(x\slr b_{i+1}|b_{i+1}\snlr A\cup B_i)). \label{eq:Unfold_Single_Operator} \end{align}

By definition of $\Sc_i$ and Lemma \ref{lem:Rewrite_Sc0} when $i=0$, we can look at $\E[\Theta\eqref{eq:Unfold_Single_Operator}]$. The term $\1_{x\slr B_{i+1}}$ yields $\Sc_{i+1}(x)$. The constant probability term gives $0$ when taking the expectation by Lemma \ref{lem:Off_Fluctuations}. It remains to understand the correction term 
\[ \E\left [\Theta \1_{b_{i+1}\snlr B_i} (\1_{x\slr b_{i+1}} -\proba(x\slr b_{i+1}|b_{i+1}\snlr A\cup B_i)) \right ]. \]
Note that if $b_{i+1}\slr A$, then the centered indicator is a function of $(C_r)_{r\in A}$. Thus, by Lemma \ref{lem:Off_Fluctuations}, the contribution of this case is null. So we may replace the first above indicator by $\1_{b_{i+1}\snlr A\cup B_i}$. Next, all variables in the display except $\Theta$ are $C_{b_{i+1}}$ measurable, so we may rewrite it as
\[ \E \left [\E[\Theta|C_{b_{i+1}}]\1_{b_{i+1}\snlr A\cup B_i} (\1_{x\slr b_{i+1}} -\proba(x\slr b_{i+1}|b_{i+1}\snlr A\cup B_i)) \right ]. \]
By the off method, we have $\E[\Theta|C_{b_{i+1}}]=\E_{\off C_{b_{i+1}}}[\Theta]=-\Phi(C_{b_{i+1}})$, whenever $(A\cup \{b\})\cap C_{b_{i+1}}=\emptyset$. We then factor out $\proba(b_{i+1}\snlr A\cup B_i)$ to turn the indicator of $b_{i+1}\snlr A\cup B_i$ into a conditioning. Then using the generic covariance formula $\E[X(Y-\E[Y])]=\Cov(X,Y)$ for any real $L^2$ random variables, together with the definition of $\kappa$, the previous expectation becomes 
\[ -\proba(b_{i+1}\snlr A\cup B_i) \Cov(\Phi(C_{b_{i+1}}),\1_{x\slr b_{i+1}}|b_{i+1}\snlr A\cup B_i )=-\proba(b_{i+1}\snlr A\cup B_i) [\kappa_{b_{i+1},A\cup B_i}\Phi](x). \]
The desired result is finally obtained by going back to $\E[\Theta\eqref{eq:Unfold_Single_Operator}]$.
\end{proof} 
We deal with the extra term in Lemma \ref{lem:OneStep_Induction} using the induction hypothesis, which requires:
\begin{lemma} \label{lem:Phi_NonDecreasing} The map $K\subseteq \Gr \mapsto \Phi(K)$ is nondecreasing. 
\end{lemma}
\begin{proof} First, by Lemma \ref{lem:Off_Fluctuations}, $\Phi(\emptyset)=-\E[\Theta]=0$, so it suffices to check that $\Phi$ is nondecreasing. Write $C_b^K$ for the cluster of $b$ after closing the edges adjacent to $K$. Define $C_A^K$ similarly. Also define $C_b^{A,K}$ as the cluster of $b$ in the initial percolation after closing all edges incident to $C_A^K$. By repeated use of Aizenman's off method, we get
\begin{align} \Phi(K)  & =\E_{\off K}\left  [\1_{b\snlr A}\left ( \E_{\off C_A}[\Fc(C_b)] -\Fc(C_b) \right ) \right ], \notag
\\ & = \E\left  [\1_{b\notin C_A^K}\left ( \E_{\off C^K_A}[\Fc(C_b)] -\Fc(C^K_b) \right ) \right ], \notag \\
& = \E\left  [\1_{b\notin C_A^K}\left (\Fc(C_b^{A,K}) -\Fc(C^K_b) \right ) \right ].\label{eq:Off_Off} 
\end{align} 
Increasing $K$ removes more edges, so it decreases $C_b^K$ and $C_A^K$. Similarly, it then increases $C_b^{A,K}$. Since $\Fc$ is nondecreasing, the difference $ \Fc(C_b^{A,K}) -\Fc(C^K_b)$ is nondecreasing. 

We claim that this difference is also nonnegative whenever $b\notin C_A^K$. Indeed, in that case $b$ is not connected to $A$ by an open path avoiding $K$, so $C_b^K\cap C_A^K=\emptyset$. As a result, for every $b'\in C_b^K$, there is a path from $b$ to $b'$ that avoids $C_A^K$, so $b'\in C_b^{A,K}$. In other words, $C^K_b\subseteq C_b^{A,K}$, and by monotonicity of $\Fc$, we get 
\[ \1_{b\notin C_A^K}(\Fc(C_b^{A,K}) -\Fc(C^K_b))\geq 0.\]
In conclusion, $\Phi(K)$ is nondecreasing since all the terms in \eqref{eq:Off_Off} are non-decreasing.
\end{proof} 
We now have all the elements to prove Theorem \ref{thm:PostOff}. 
\begin{proof} Repeatedly applying Lemma \ref{lem:OneStep_Induction} gives, by induction, for every $0\leq i\leq k$ and $x\in \Gr$, writing $B_i=\{b,b_1,\dots, b_i\}$,
\[ [\pi_{A\cup \{b\},(b_1,b_2,\dots,b_i)} \Sc](x)=\Sc_{i}(x)+\sum_{j=0}^{i-1}\proba(b_{j+1}\snlr A\cup B_j)\Big [\pi_{A\cup B_{j+1},(b_{j+2},\dots,b_i)} [\kappa_{b_{j+1},A\cup B_j}\Phi] \Big ](x) .\]
By Lemma \ref{lem:Phi_NonDecreasing}, $\Phi$ is nondecreasing, so by the induction hypothesis used with $A\cup B_j$, $b_{j+1}$, and the sequence $(b_{j+2},\dots b_i)$, which is shorter than $\Bc$, each of the terms in the above sum is nonnegative. Taking $i=k$ therefore gives
\[ \pi_{A\cup\{b\}, \Bc}\Sc(x)\geq \Sc_k(x). \]
Moreover, by the same argument as in the proof of Lemma \ref{lem:Rewrite_Sc0}, we may rewrite $ \Sc_k(x)$ as 
\[ \Sc_k(x)=\E[\Cov_{\off C_A}(\Fc(C_b);\1_{x\slr B_k});b\snlr A]. \]
By the Harris inequality, the covariance above is a.s. nonnegative, so $\Sc_k(x)\geq 0$. Therefore, $$\pi_{A\cup\{b\}, \Bc}\Sc(x)\geq 0. \qedhere$$ 
\end{proof} 
\subsection{The Gibbs argument}
We now deduce Theorem \ref{thm:Projection_Covariance} from the fact that Theorem \ref{thm:PostOff} holds for every $\Fc$.  We first rewrite $ [\pi_{A\cup \{b\},\Bc}\circ \kappa_{b,A}]\Fc$. To this end, define the scalar product $\langle ,\rangle$ for functions $\Fc,\Gc:\Comp\mapsto \R$ by 
\[ \langle \Fc,\Gc\rangle = \Cov(\Fc(C_b),\Gc(C_b)|b\snlr A).\]
Then, for every $u\in \Gr$, let $\1_u$ denote the function $x\in \Gr\mapsto \1_{x=u}$ and let $\1_{u\in} :C\in \Comp \mapsto \1_{u\in C}$. For $C\subseteq\Gr$, let $\1_{\in C}:u\in \Gr\mapsto \1_{u\in C}$. We may rewrite the operator $\kappa_{b,A}$ as $\kappa_{b,A}\Fc:x\mapsto \langle \Fc,\1_{x\in} \rangle.$
For every $\Gc:\Gr\mapsto \R$, we have $\Gc=\sum_{u\in \Gr} \1_{u} \cdot \Gc(u)$. Using this formula for the functions $\kappa_{b,A}\Fc$ and using the linearity of $\pi_{A\cup \{b\},\Bc}$, we deduce for every $x\in \Gr$ that 
\begin{align} [\pi_{A\cup \{b\},\Bc}\circ \kappa_{b,A}]\Fc(x)
 & =\sum_{u\in \Gr}[\pi_{A\cup \{b\},\Bc}\1_u](x)\cdot  [\kappa_{b,A}\Fc](u) \notag
 \\ & =\sum_{u\in \Gr}[\pi_{A\cup \{b\},\Bc}\1_u](x) \cdot \langle \Fc ;\1_{u\in} \rangle \notag 
 \\ & =\langle \Fc ;\Hc_x \rangle, \label{eq:Rewrite_Scalar_Product}
 \end{align}
writing for every $C\in \Comp$, \[ \Hc_x(C)=\sum_{u\in \Gr}  [\pi_{A\cup \{b\},\Bc}\1_u](x)\1_{u\in C}= [\pi_{A\cup \{b\},\Bc} \1_{\in C}](x),\]
reusing the linearity of $\pi_{A\cup \{b\},\Bc}$ and the identity $\sum_{u\in \Gr} \1_{u}\cdot \Gc(u)=\Gc$ with $\Gc=\1_{\in C}$.

Next, by the law of total covariance, we get 
\begin{align} [\pi_{A\cup \{b\},\Bc}\circ \kappa_{b,A}]\Fc(x)
  = & \E \left [\Cov \left (\Fc(C_b);\Hc_x(C_b) \middle |C_A\right ) \middle |b\snlr A\right ] \notag
\\ &  +\Cov\left (\E [\Fc(C_b)|C_A] , \E [\Hc_x(C_b)|C_A] \middle |b\snlr A \right ).  \label{eq:Split_Gibs}\end{align}
We rewrite and deal with the two terms on the right-hand side separately. By the off method, the first expectation equals 
\[ \E \left [\Cov_{\off C_A} \left (\Fc(C_b);\Hc_x(C_b)\right ) \middle |b\snlr A\right ]. \]
Then, for $K\subseteq \Gr$, consider the scalar product defined by $\langle \Fc,\Gc\rangle_K:=\Cov_{\off K} \left (\Fc(C_b);\Gc(C_b)\right )$ for every $\Fc,\Gc:\Comp\mapsto \R$, and the operator $\kappa_{b,\off K}$ defined by $\kappa_{b,\off K}\Fc(x):=\Cov_{\off K} (\Fc(C_b);\1_{x\slr b} )$. Using the same linear computations as for $[\pi_{A\cup \{b\},\Bc}\circ \kappa_{b,A}]\Fc(x)$, with the new scalar product and $\kappa$ operator, we can rewrite the above as
\[ \E\left [\big [\pi_{A\cup \{b\},\Bc}\circ \kappa_{b,\off C_A} \big ]\Fc(x) \middle |b\snlr A  \right ],  \]
By linearity of $\pi_{A\cup \{b\},\Bc}$ and the definition of $\kappa_{b,\off C_A}$, this expectation equals 
\[ \pi_{A\cup \{b\},\Bc} \left [y\mapsto \E \left [\Cov_{\off C_A}(\Fc(C_b),\1_{y\slr b}) \middle | b\snlr A \right ]  \right ](x).\]
By multiplying the above by $\proba(b\snlr A)$, we recognize exactly the function $\Sc$ of Theorem \ref{thm:PostOff}, so  the last display is nonnegative.

For the second term of \eqref{eq:Split_Gibs}, since taking the conditional expectation is an orthogonal projection,
\begin{align*} \Cov\left (\E [\Fc(C_b)|C_A] , \E [\Hc_x(C_b)|C_A] \middle |b\snlr A \right ) & = \Cov\left (\E [\Fc(C_b)|C_A] , \Hc_x(C_b) \middle |b\snlr A \right )
\\ & = \Cov\left (\E[\E [\Fc(C_b)|C_A]|C_b] , \Hc_x(C_b) \middle |b\snlr A \right )
\\ & = \langle \Tf[\Fc],\Hc_x\rangle, 
\end{align*}
where $\Tf$ is the linear operator defined such that for every $\Fc:\Comp\mapsto \R$ and $C\in \Comp$ we have
\[ \Tf[\Fc](C)=\E_{\off C}[\E_{\off C_A} [\Fc(C_b)]]. \] 
In summary, we get the following result.
\begin{proposition} \label{pro:Gibs_OneStep} For every $\Fc:\Comp\mapsto \R$ nondecreasing, 
 $$\langle \Fc, \Hc_x\rangle \geq  \langle \Tf \Fc, \Hc_x\rangle.$$
\end{proposition} 
To use the previous proposition repeatedly, we need the following property.
\begin{lemma} \label{lem:Gibs_Monotony} For every $\Fc:\Comp\mapsto \R$ nondecreasing, the function $\Tf \Fc:\Comp\mapsto \R$ is nondecreasing.
\end{lemma} 
\begin{proof} Note that $\Tf$ is obtained as the conditional expectation given an initial cluster $C$, after first resampling $C_A$ and then $C_b$. The result follows directly from the fact that the larger $C$ is, the smaller $C_A$ is when sampled, and so the larger $C_b$ becomes when resampled.
\end{proof} 
By Lemma \ref{lem:Gibs_Monotony}, we may repeatedly use Proposition \ref{pro:Gibs_OneStep} to get for every $n\geq 0$, 
\[ \langle \Fc, \Hc_x\rangle \geq \langle \Tf^n \Fc, \Hc_x\rangle.\]
Thus, by \eqref{eq:Rewrite_Scalar_Product}, the proof of Theorem \ref{thm:Projection_Covariance} is complete as long as we prove that as $n\to \infty$,
\[ \langle \Tf^n \Fc, \Hc_x\rangle\to 0.\]
To this end, note that $\Tf^n \Fc$ is the conditional expectation starting from an initial configuration for $C_b$, obtained by resampling $C_A$ and then $C_b$ alternately in this order. (This mechanism is the main reason why we call this subsection the Gibbs step.) So the above convergence follows from the fact that the chain mixes. There are many ways to see this. A simple reason is that the Gibbs chain has a regenerative event that occurs at each step with probability bounded below: $C_b$ can always be $\{b\}$.

 \end{document}